\documentclass{amsart}
\usepackage{amsmath}
\usepackage{amsfonts}

\newtheorem{theorem}{Theorem}
\theoremstyle{plain}

\newtheorem{definition}{Definition}

\newtheorem{lemma}{Lemma}

\newtheorem{remark}{Remark}

\numberwithin{equation}{section}
\input{tcilatex}

\begin{document}
\title[Hermite-Hadamard-Fejer type inequalities]{Hermite-Hadamard-Fejer type
inequalities for convex functions via fractional integrals}
\author{\.{I}mdat \.{I}\c{s}can}
\address{Department of Mathematics, Faculty of Sciences and Arts, Giresun
University, Giresun, Turkey}
\email{imdat.iscan@giresun.edu.tr; imdati@yahoo.com}
\subjclass[2000]{ 26A51, 26A33, 26D10. }
\keywords{Hermite-Hadamard inequality, Hermite-Hadamard-Fejer inequality,
Riemann-Liouville fractional integral, convex function.}

\begin{abstract}
In this paper, firstly we have established Hermite--Hadamard-Fej\'{e}r
inequality for fractional integrals. Secondly, an integral identity and some
Hermite-Hadamard-Fejer type integral inequalities for the fractional
integrals have been obtained. The some results presented here would provide
extensions of those given in earlier works.
\end{abstract}

\maketitle

\section{Introduction}

Let $f:I\subseteq \mathbb{R\rightarrow R}$ be a convex function defined on
the interval $I$ of real numbers and $a,b\in I$ with $a<b$. The inequality 
\begin{equation}
f\left( \frac{a+b}{2}\right) \leq \frac{1}{b-a}\int_{a}^{b}f(x)dx\leq \frac{%
f(a)+f(b)}{2}  \label{1-1}
\end{equation}%
is well known in the literature as Hermite-Hadamard's inequality \cite{H93}.

The most well-known inequalities related to the integral mean of a convex
function $f$ are the Hermite Hadamard inequalities or its weighted versions,
the so-called Hermite-Hadamard-Fej\'{e}r inequalities.

In \cite{F06}, Fej\'{e}r established the following Fej\'{e}r inequality
which is the weighted generalization of Hermite-Hadamard inequality (\ref%
{1-1}):

\begin{theorem}
\label{1.2}Let $f:\left[ a,b\right] \mathbb{\rightarrow R}$ be convex
function. Then the inequality%
\begin{equation}
f\left( \frac{a+b}{2}\right) \int_{a}^{b}g(x)dx\leq \frac{1}{b-a}%
\int_{a}^{b}f(x)g(x)dx\leq \frac{f(a)+f(b)}{2}\int_{a}^{b}g(x)dx  \label{1-2}
\end{equation}%
holds, where $g:\left[ a,b\right] \mathbb{\rightarrow R}$ is
nonnegative,integrable and symmetric to $(a+b)/2.$
\end{theorem}

For some results which generalize, improve, and extend the inequalities (\ref%
{1-1}) and (\ref{1-2}) see \cite{BV09,I13a,I13aa,I13ab,S12a,TYH11}.

We give some necessary definitions and mathematical preliminaries of
fractional calculus theory which are used throughout this paper.

\begin{definition}
Let $f\in L\left[ a,b\right] $. The Riemann-Liouville integrals $%
J_{a+}^{\alpha }f$ and $J_{b-}^{\alpha }f$ of oder $\alpha >0$ with $a\geq 0$
are defined by

\begin{equation*}
J_{a+}^{\alpha }f(x)=\frac{1}{\Gamma (\alpha )}\int_{a}^{x}\left( x-t\right)
^{\alpha -1}f(t)dt,\ x>a
\end{equation*}

and

\begin{equation*}
J_{b-}^{\alpha }f(x)=\frac{1}{\Gamma (\alpha )}\int_{x}^{b}\left( t-x\right)
^{\alpha -1}f(t)dt,\ x<b
\end{equation*}%
respectively, where $\Gamma (\alpha )$ is the Gamma function defined by $%
\Gamma (\alpha )=$ $\dint\limits_{0}^{\infty }e^{-t}t^{\alpha -1}dt$ and $%
J_{a+}^{0}f(x)=J_{b-}^{0}f(x)=f(x).$
\end{definition}

Because of the wide application of Hermite-Hadamard type inequalities and
fractional integrals, many researchers extend their studies to
Hermite-Hadamard type inequalities involving fractional integrals not
limited to integer integrals. Recently, more and more Hermite-Hadamard
inequalities involving fractional integrals have been obtained for different
classes of functions; see \cite{D10,I13b,I13d,SSYB13,S12,WFZ12,WZZ13}.

In \cite{SSYB13}, Sar\i kaya et. al. represented Hermite--Hadamard's
inequalities in fractional integral forms as follows.

\begin{theorem}
\label{1.3}Let $f:\left[ a,b\right] \rightarrow 
\mathbb{R}
$ be a positive function with $0\leq a<b$ and $f\in L\left[ a,b\right] $. If 
$f$ is a convex function on $\left[ a,b\right] $, then the following
inequalities for fractional integrals hold%
\begin{equation}
f\left( \frac{a+b}{2}\right) \leq \frac{\Gamma (\alpha +1)}{2\left(
b-a\right) ^{\alpha }}\left[ J_{a+}^{\alpha }f(b)+J_{b-}^{\alpha }f(a)\right]
\leq \frac{f(a)+f(b)}{2}  \label{1-3}
\end{equation}%
with $\alpha >0.$
\end{theorem}

In \cite{SSYB13} some Hermite-Hadamard type integral inequalities for
fractional integral were proved using the following lemma.

\begin{lemma}
\label{1.4}Let $f:\left[ a,b\right] \rightarrow 
\mathbb{R}
$ be a differentiable mapping on $\left( a,b\right) $ with $a<b.$ If $%
f^{\prime }\in L\left[ a,b\right] $ then the following equality for
fractional integrals holds:%
\begin{equation}
\frac{f(a)+f(b)}{2}-\frac{\Gamma (\alpha +1)}{2\left( b-a\right) ^{\alpha }}%
\left[ J_{a+}^{\alpha }f(b)+J_{b-}^{\alpha }f(a)\right] =\frac{b-a}{2}%
\int_{0}^{1}\left[ \left( 1-t\right) ^{\alpha }-t^{\alpha }\right] f^{\prime
}\left( ta+(1-t)b\right) dt.  \label{1-4}
\end{equation}
\end{lemma}

\begin{theorem}
\label{1.5}Let $f:\left[ a,b\right] \rightarrow 
\mathbb{R}
$ be a differentiable mapping on $\left( a,b\right) $ with $a<b.$ If $%
\left\vert f^{\prime }\right\vert $ is convex on $\left[ a,b\right] $ then
the following inequality for fractional integrals holds:%
\begin{equation}
\left\vert \frac{f(a)+f(b)}{2}-\frac{\Gamma (\alpha +1)}{2\left( b-a\right)
^{\alpha }}\left[ J_{a+}^{\alpha }f(b)+J_{b-}^{\alpha }f(a)\right]
\right\vert \leq \frac{b-a}{2\left( \alpha +1\right) }\left( 1-\frac{1}{%
2^{\alpha }}\right) \left[ \left\vert f^{\prime }\left( a\right) \right\vert
+\left\vert f^{\prime }\left( b\right) \right\vert \right] .  \label{1-5}
\end{equation}
\end{theorem}

\begin{lemma}[\protect\cite{PBM81,WZZ13}]
\label{1.6}For $0<\alpha \leq 1$ and $0\leq a<b$, we have%
\begin{equation*}
\left\vert a^{\alpha }-b^{\alpha }\right\vert \leq \left( b-a\right)
^{\alpha }.
\end{equation*}
\end{lemma}

In this paper, we firstly represented Hermite-Hadamard-Fej\'{e}r inequality
in fractional integral forms which is the weighted generalization of
Hermite-Hadamard inequality (\ref{1-3}). Secondly, we obtained some new
inequalities connected with the right-hand side of Hermite-Hadamard-Fej\'{e}%
r type integral inequality for the fractional integrals.

\section{Main results}

Throughout this section, let $\left\Vert g\right\Vert _{\infty }=\sup_{t\in %
\left[ a,b\right] }\left\vert g(x)\right\vert $, for the continuous function 
$g:\left[ a,b\right] \mathbb{\rightarrow 
\mathbb{R}
}$.

\begin{lemma}
\label{2.1}If $g:\left[ a,b\right] \mathbb{\rightarrow R}$ is integrable and
symmetric to $(a+b)/2$ with $a<b$, then 
\begin{equation*}
J_{a+}^{\alpha }g(b)=J_{b-}^{\alpha }g(a)=\frac{1}{2}\left[ J_{a+}^{\alpha
}g(b)+J_{b-}^{\alpha }g(a)\right]
\end{equation*}%
with $\alpha >0.$
\end{lemma}

\begin{proof}
Since $g$ is symmetric to $(a+b)/2,$ we have $g\left( a+b-x\right) =g(x)$,
for all $x\in \left[ a,b\right] .$ Hence, in the following integral setting $%
x=tb+(1-t)a$ and $dx=\left( b-a\right) dt$ gives 
\begin{eqnarray*}
J_{a+}^{\alpha }g(b) &=&\frac{1}{\Gamma (\alpha )}\dint\limits_{a}^{b}\left(
b-x\right) ^{\alpha -1}g(x)dx \\
&=&\frac{1}{\Gamma (\alpha )}\dint\limits_{a}^{b}\left( x-a\right) ^{\alpha
-1}g(a+b-x)dx \\
&=&\frac{1}{\Gamma (\alpha )}\dint\limits_{a}^{b}\left( x-a\right) ^{\alpha
-1}g(x)dx=J_{b-}^{\alpha }g(a).
\end{eqnarray*}%
This completes the proof.
\end{proof}

\begin{theorem}
\label{2.2}Let $f:\left[ a,b\right] \mathbb{\rightarrow R}$ be convex
function with $a<b$ and $f\in L\left[ a,b\right] $. If $g:\left[ a,b\right] 
\mathbb{\rightarrow R}$ is nonnegative,integrable and symmetric to $(a+b)/2$%
, then the following inequalities for fractional integrals hold%
\begin{eqnarray}
f\left( \frac{a+b}{2}\right) \left[ J_{a+}^{\alpha }g(b)+J_{b-}^{\alpha }g(a)%
\right] &\leq &\left[ J_{a+}^{\alpha }\left( fg\right) (b)+J_{b-}^{\alpha
}\left( fg\right) (a)\right]  \label{2-2} \\
&\leq &\frac{f(a)+f(b)}{2}\left[ J_{a+}^{\alpha }g(b)+J_{b-}^{\alpha }g(a)%
\right]  \notag
\end{eqnarray}%
with $\alpha >0.$
\end{theorem}

\begin{proof}
Since $f$ is a convex function on $\left[ a,b\right] $, we have for all $%
t\in \left[ 0,1\right] $%
\begin{eqnarray}
f\left( \frac{a+b}{2}\right) &=&f\left( \frac{ta+(1-t)b+tb+(1-t)a}{2}\right)
\notag \\
&\leq &\frac{f\left( ta+(1-t)b\right) +f\left( tb+(1-t)a\right) }{2}.
\label{2-2a}
\end{eqnarray}%
Multiplying both sides of (\ref{2-2a}) by $2t^{\alpha -1}g\left(
tb+(1-t)a\right) $ then integrating the resulting inequality with respect to 
$t$ over $\left[ 0,1\right] $, we obtain%
\begin{eqnarray*}
&&2f\left( \frac{a+b}{2}\right) \int_{0}^{1}t^{\alpha -1}g\left(
tb+(1-t)a\right) dt \\
&\leq &\int_{0}^{1}t^{\alpha -1}\left[ f\left( ta+(1-t)b\right) +f\left(
tb+(1-t)a\right) \right] g\left( tb+(1-t)a\right) dt \\
&=&\int_{0}^{1}t^{\alpha -1}f\left( ta+(1-t)b\right) g\left(
tb+(1-t)a\right) dt+\int_{0}^{1}t^{\alpha -1}f\left( tb+(1-t)a\right)
g\left( tb+(1-t)a\right) dt.
\end{eqnarray*}%
Setting $x=tb+(1-t)a$, and $dx=\left( b-a\right) dt$ gives%
\begin{eqnarray*}
&&\frac{2}{\left( b-a\right) ^{\alpha }}f\left( \frac{a+b}{2}\right)
\int_{a}^{b}\left( x-a\right) ^{\alpha -1}g\left( x\right) dx \\
&\leq &\frac{1}{\left( b-a\right) ^{\alpha }}\left\{ \int_{a}^{b}\left(
x-a\right) ^{\alpha -1}f\left( a+b-x\right) g\left( x\right)
dx+\int_{0}^{1}\left( x-a\right) ^{\alpha -1}f\left( x\right) g\left(
x\right) dx\right\} \\
&=&\frac{1}{\left( b-a\right) ^{\alpha }}\left\{ \int_{a}^{b}\left(
b-x\right) ^{\alpha -1}f\left( x\right) g\left( a+b-x\right)
dx+\int_{0}^{1}\left( x-a\right) ^{\alpha -1}f\left( x\right) g\left(
x\right) dx\right\} \\
&=&\frac{1}{\left( b-a\right) ^{\alpha }}\left\{ \int_{a}^{b}\left(
b-x\right) ^{\alpha -1}f\left( x\right) g\left( x\right)
dx+\int_{0}^{1}\left( x-a\right) ^{\alpha -1}f\left( x\right) g\left(
x\right) dx\right\} .
\end{eqnarray*}%
Therefore, by Lemma \ref{2.1} we have%
\begin{equation*}
\frac{\Gamma (\alpha )}{\left( b-a\right) ^{\alpha }}f\left( \frac{a+b}{2}%
\right) \left[ J_{a+}^{\alpha }g(b)+J_{b-}^{\alpha }g(a)\right] \leq \frac{%
\Gamma (\alpha )}{\left( b-a\right) ^{\alpha }}\left[ J_{a+}^{\alpha }\left(
fg\right) (b)+J_{b-}^{\alpha }\left( fg\right) (a)\right]
\end{equation*}

and the first inequality is proved.

For the proof of the second inequality in (\ref{2-2}) we first note that if $%
f$ is a convex function, then, for all $t\in \left[ 0,1\right] $, it yields%
\begin{equation}
f\left( ta+(1-t)b\right) +f\left( tb+(1-t)a\right) \leq f(a)+f(b).
\label{2-2b}
\end{equation}%
Then multiplying both sides of (\ref{2-2b}) by $2t^{\alpha -1}g\left(
tb+(1-t)a\right) $ and integrating the resulting inequality with respect to $%
t$ over $\left[ 0,1\right] $, we obtain%
\begin{eqnarray*}
&&\int_{0}^{1}t^{\alpha -1}f\left( ta+(1-t)b\right) g\left( tb+(1-t)a\right)
dt+\int_{0}^{1}t^{\alpha -1}f\left( tb+(1-t)a\right) g\left(
tb+(1-t)a\right) dt \\
&\leq &\left[ f(a)+f(b)\right] \int_{0}^{1}t^{\alpha -1}g\left(
tb+(1-t)a\right) dt
\end{eqnarray*}%
i.e.%
\begin{equation*}
\frac{\Gamma (\alpha )}{\left( b-a\right) ^{\alpha }}\left[ J_{a+}^{\alpha
}\left( fg\right) (b)+J_{b-}^{\alpha }\left( fg\right) (a)\right] \leq \frac{%
\Gamma (\alpha )}{\left( b-a\right) ^{\alpha }}\left( \frac{f(a)+f(b)}{2}%
\right) \left[ J_{a+}^{\alpha }g(b)+J_{b-}^{\alpha }g(a)\right]
\end{equation*}%
The proof is completed.
\end{proof}

\begin{remark}
In Theorem \ref{2.2},

(i) if we take $\alpha =1$, then inequality (\ref{2-2}) becomes inequality (%
\ref{1-2}) of Theorem \ref{1.2}.

(ii) if we take $g(x)=1$, then inequality (\ref{2-2}) becomes inequality (%
\ref{1-3}) of Theorem \ref{1.3}.
\end{remark}

\begin{lemma}
\label{2.3}Let $f:\left[ a,b\right] \mathbb{\rightarrow R}$ be a
differentiable mapping on $\left( a,b\right) $ with $a<b$ and $f^{\prime
}\in L\left[ a,b\right] $. If $g:\left[ a,b\right] \mathbb{\rightarrow R}$
is integrable and symmetric to $(a+b)/2$ then the following equality for
fractional integrals holds%
\begin{eqnarray}
&&\left( \frac{f(a)+f(b)}{2}\right) \left[ J_{a+}^{\alpha
}g(b)+J_{b-}^{\alpha }g(a)\right] -\left[ J_{a+}^{\alpha }\left( fg\right)
(b)+J_{b-}^{\alpha }\left( fg\right) (a)\right]  \label{2-3} \\
&=&\frac{1}{\Gamma (\alpha )}\int_{a}^{b}\left[ \int_{a}^{t}\left(
b-s\right) ^{\alpha -1}g(s)ds-\int_{t}^{b}\left( s-a\right) ^{\alpha
-1}g(s)ds\right] f^{\prime }(t)dt  \notag
\end{eqnarray}%
with $\alpha >0.$
\end{lemma}

\begin{proof}
It suffices to note that%
\begin{eqnarray*}
I &=&\int_{a}^{b}\left[ \int_{a}^{t}\left( b-s\right) ^{\alpha
-1}g(s)ds-\int_{t}^{b}\left( s-a\right) ^{\alpha -1}g(s)ds\right] f^{\prime
}(t)dt \\
&=&\int_{a}^{b}\left( \int_{a}^{t}\left( b-s\right) ^{\alpha
-1}g(s)ds\right) f^{\prime }(t)dt+\int_{a}^{b}\left( -\int_{t}^{b}\left(
s-a\right) ^{\alpha -1}g(s)ds\right) f^{\prime }(t)dt \\
&=&I_{1}+I_{2}.
\end{eqnarray*}%
By integration by parts and Lemma \ref{2.1}\ we get%
\begin{eqnarray*}
I_{1} &=&\left. \left( \int_{a}^{t}\left( b-s\right) ^{\alpha
-1}g(s)ds\right) f(t)\right\vert _{a}^{b}-\int_{a}^{b}\left( b-t\right)
^{\alpha -1}g(t)f(t)dt \\
&=&\left( \int_{a}^{b}\left( b-s\right) ^{\alpha -1}g(s)ds\right)
f(b)-\int_{a}^{b}\left( b-t\right) ^{\alpha -1}(fg)(t)dt \\
&=&\Gamma (\alpha )\left[ f(b)J_{a+}^{\alpha }g(b)-J_{a+}^{\alpha }(fg)(b)%
\right] \\
&=&\Gamma (\alpha )\left[ \frac{f(b)}{2}\left[ J_{a+}^{\alpha
}g(b)+J_{b-}^{\alpha }g(a)\right] -J_{a+}^{\alpha }(fg)(b)\right]
\end{eqnarray*}%
and similarly%
\begin{eqnarray*}
I_{2} &=&\left. \left( -\int_{t}^{b}\left( s-a\right) ^{\alpha
-1}g(s)ds\right) f(t)\right\vert _{a}^{b}-\int_{a}^{b}\left( t-a\right)
^{\alpha -1}g(t)f(t)dt \\
&=&\left( \int_{a}^{b}\left( s-a\right) ^{\alpha -1}g(s)ds\right)
f(a)-\int_{a}^{b}\left( t-a\right) ^{\alpha -1}(fg)(t)dt \\
&=&\Gamma (\alpha )\left[ \frac{f(a)}{2}\left[ J_{a+}^{\alpha
}g(b)+J_{b-}^{\alpha }g(a)\right] -J_{b-}^{\alpha }\left( fg\right) (a)%
\right] .
\end{eqnarray*}%
Thus, we can write%
\begin{equation*}
I=I_{1}+I_{2}=\Gamma (\alpha )\left\{ \left( \frac{f(a)+f(b)}{2}\right) %
\left[ J_{a+}^{\alpha }g(b)+J_{b-}^{\alpha }g(a)\right] -\left[
J_{a+}^{\alpha }\left( fg\right) (b)+J_{b-}^{\alpha }\left( fg\right) (a)%
\right] \right\} .
\end{equation*}%
Multiplying the both sides by $\left( \Gamma (\alpha )\right) ^{-1}$ we
obtain (\ref{2-3}) which completes the proof.
\end{proof}

\begin{remark}
In Lemma \ref{2.3}, if we take $g(x)=1$, then equality (\ref{2-3}) becomes
equality (\ref{1-4}) of Lemma \ref{1.4}.
\end{remark}

\begin{theorem}
\label{2.4}Let $f:I\subseteq \mathbb{R\rightarrow R}$ be a differentiable
mapping on $I^{\circ }$ and $f^{\prime }\in L\left[ a,b\right] $ with $a<b$.
If $\left\vert f^{\prime }\right\vert $ is convex on $\left[ a,b\right] $
and $g:\left[ a,b\right] \mathbb{\rightarrow R}$ is continuous and symmetric
to $(a+b)/2$, then the following inequality for fractional integrals holds%
\begin{eqnarray}
&&\left\vert \left( \frac{f(a)+f(b)}{2}\right) \left[ J_{a+}^{\alpha
}g(b)+J_{b-}^{\alpha }g(a)\right] -\left[ J_{a+}^{\alpha }\left( fg\right)
(b)+J_{b-}^{\alpha }\left( fg\right) (a)\right] \right\vert  \notag \\
&\leq &\frac{\left( b-a\right) ^{\alpha +1}\left\Vert g\right\Vert _{\infty }%
}{\left( \alpha +1\right) \Gamma (\alpha +1)}\left( 1-\frac{1}{2^{\alpha }}%
\right) \left[ \left\vert f^{\prime }\left( a\right) \right\vert +\left\vert
f^{\prime }\left( b\right) \right\vert \right]  \label{2-4}
\end{eqnarray}%
with $\alpha >0.$
\end{theorem}

\begin{proof}
From Lemma \ref{2.3} we have%
\begin{eqnarray}
&&\left\vert \left( \frac{f(a)+f(b)}{2}\right) \left[ J_{a+}^{\alpha
}g(b)+J_{b-}^{\alpha }g(a)\right] -\left[ J_{a+}^{\alpha }\left( fg\right)
(b)+J_{b-}^{\alpha }\left( fg\right) (a)\right] \right\vert  \notag \\
&\leq &\frac{1}{\Gamma (\alpha )}\int_{a}^{b}\left\vert \int_{a}^{t}\left(
b-s\right) ^{\alpha -1}g(s)ds-\int_{t}^{b}\left( s-a\right) ^{\alpha
-1}g(s)ds\right\vert \left\vert f^{\prime }(t)\right\vert dt.  \label{2-4a}
\end{eqnarray}%
Since $\left\vert f^{\prime }\right\vert $ is convex on $\left[ a,b\right] $%
, we know that for $t\in \left[ a,b\right] $%
\begin{equation}
\left\vert f^{\prime }(t)\right\vert =\left\vert f^{\prime }\left( \frac{b-t%
}{b-a}a+\frac{t-a}{b-a}b\right) \right\vert \leq \frac{b-t}{b-a}\left\vert
f^{\prime }\left( a\right) \right\vert +\frac{t-a}{b-a}\left\vert f^{\prime
}\left( b\right) \right\vert ,  \label{2-4b}
\end{equation}%
and since $g:\left[ a,b\right] \mathbb{\rightarrow R}$ is symmetric to $%
(a+b)/2$ we write%
\begin{equation*}
\int_{t}^{b}\left( s-a\right) ^{\alpha -1}g(s)ds=\int_{a}^{a+b-t}\left(
b-s\right) ^{\alpha -1}g(a+b-s)ds=\int_{a}^{a+b-t}\left( b-s\right) ^{\alpha
-1}g(s)ds,
\end{equation*}%
then we have%
\begin{eqnarray}
&&\left\vert \int_{a}^{t}\left( b-s\right) ^{\alpha
-1}g(s)ds-\int_{t}^{b}\left( s-a\right) ^{\alpha -1}g(s)ds\right\vert  \notag
\\
&=&\left\vert \int_{t}^{a+b-t}\left( b-s\right) ^{\alpha -1}g(s)ds\right\vert
\notag \\
&\leq &\left\{ 
\begin{array}{cc}
\int_{t}^{a+b-t}\left\vert \left( b-s\right) ^{\alpha -1}g(s)\right\vert ds
& t\in \left[ a,\frac{a+b}{2}\right] \\ 
\int_{a+b-t}^{t}\left\vert \left( b-s\right) ^{\alpha -1}g(s)\right\vert ds
& t\in \left[ \frac{a+b}{2},b\right]%
\end{array}%
\right. .  \label{2-4c}
\end{eqnarray}%
A combination of (\ref{2-4a}), (\ref{2-4b}) and (\ref{2-4c}), we get 
\begin{eqnarray}
&&\left\vert \left( \frac{f(a)+f(b)}{2}\right) \left[ J_{a+}^{\alpha
}g(b)+J_{b-}^{\alpha }g(a)\right] -\left[ J_{a+}^{\alpha }\left( fg\right)
(b)+J_{b-}^{\alpha }\left( fg\right) (a)\right] \right\vert  \notag \\
&\leq &\frac{1}{\Gamma (\alpha )}\int_{a}^{\frac{a+b}{2}}\left(
\int_{t}^{a+b-t}\left\vert \left( b-s\right) ^{\alpha -1}g(s)\right\vert
ds\right) \left( \frac{b-t}{b-a}\left\vert f^{\prime }\left( a\right)
\right\vert +\frac{t-a}{b-a}\left\vert f^{\prime }\left( b\right)
\right\vert \right) dt  \notag \\
&&+\frac{1}{\Gamma (\alpha )}\int_{\frac{a+b}{2}}^{b}\left(
\int_{a+b-t}^{t}\left\vert \left( b-s\right) ^{\alpha -1}g(s)\right\vert
ds\right) \left( \frac{b-t}{b-a}\left\vert f^{\prime }\left( a\right)
\right\vert +\frac{t-a}{b-a}\left\vert f^{\prime }\left( b\right)
\right\vert \right) dt  \notag \\
&\leq &\frac{\left\Vert g\right\Vert _{\infty }}{\left( b-a\right) \Gamma
(\alpha +1)}\left\{ \int_{a}^{\frac{a+b}{2}}\left[ \left( b-t\right)
^{\alpha }-\left( t-a\right) ^{\alpha }\right] \left( \left( b-t\right)
\left\vert f^{\prime }\left( a\right) \right\vert +\left( t-a\right)
\left\vert f^{\prime }\left( b\right) \right\vert \right) dt\right.  \notag
\\
&&\left. +\int_{\frac{a+b}{2}}^{b}\left[ \left( t-a\right) ^{\alpha }-\left(
b-t\right) ^{\alpha }\right] \left( \left( b-t\right) \left\vert f^{\prime
}\left( a\right) \right\vert +\left( t-a\right) \left\vert f^{\prime }\left(
b\right) \right\vert \right) dt\right\}  \label{2-4d}
\end{eqnarray}%
Since%
\begin{eqnarray}
&&\int_{a}^{\frac{a+b}{2}}\left[ \left( b-t\right) ^{\alpha }-\left(
t-a\right) ^{\alpha }\right] \left( b-t\right) dt  \notag \\
&=&\int_{\frac{a+b}{2}}^{b}\left[ \left( t-a\right) ^{\alpha }-\left(
b-t\right) ^{\alpha }\right] \left( t-a\right) dt  \notag \\
&=&\frac{\left( b-a\right) ^{\alpha +2}}{\left( \alpha +1\right) }\left( 
\frac{\alpha +1}{\alpha +2}-\frac{1}{2^{\alpha +1}}\right)  \label{2-4e}
\end{eqnarray}%
and 
\begin{eqnarray}
&&\int_{a}^{\frac{a+b}{2}}\left[ \left( b-t\right) ^{\alpha }-\left(
t-a\right) ^{\alpha }\right] \left( t-a\right) dt  \notag \\
&=&\int_{\frac{a+b}{2}}^{b}\left[ \left( t-a\right) ^{\alpha }-\left(
b-t\right) ^{\alpha }\right] \left( b-t\right) dt  \notag \\
&=&\frac{\left( b-a\right) ^{\alpha +2}}{\left( \alpha +1\right) }\left( 
\frac{1}{\alpha +2}-\frac{1}{2^{\alpha +1}}\right)  \label{2-4f}
\end{eqnarray}%
Hence, if we use (\ref{2-4e}) and (\ref{2-4f}) in (\ref{2-4d}), we obtain
the desired result. This completes the proof.
\end{proof}

\begin{remark}
In Theorem \ref{2.4}, if we take $g(x)=1$, then equality (\ref{2-4}) becomes
equality (\ref{1-5}) of Theorem \ref{1.5}.
\end{remark}

\begin{theorem}
Let $f:I\subseteq \mathbb{R\rightarrow R}$ be a differentiable mapping on $%
I^{\circ }$ and $f^{\prime }\in L\left[ a,b\right] $ with $a<b$. If $%
\left\vert f^{\prime }\right\vert ^{q},q>1,$ is convex on $\left[ a,b\right] 
$ and $g:\left[ a,b\right] \mathbb{\rightarrow R}$ is continuous and
symmetric to $(a+b)/2$, then the following inequality for fractional
integrals holds%
\begin{equation}
\left\vert \left( \frac{f(a)+f(b)}{2}\right) \left[ J_{a+}^{\alpha
}g(b)+J_{b-}^{\alpha }g(a)\right] -\left[ J_{a+}^{\alpha }\left( fg\right)
(b)+J_{b-}^{\alpha }\left( fg\right) (a)\right] \right\vert  \label{2-5}
\end{equation}%
\begin{equation*}
\leq \frac{2\left( b-a\right) ^{\alpha +1}\left\Vert g\right\Vert _{\infty }%
}{\left( b-a\right) ^{1/q}(\alpha +1)\Gamma (\alpha +1)}\left( 1-\frac{1}{%
2^{\alpha }}\right) \left( \frac{\left\vert f^{\prime }\left( a\right)
\right\vert ^{q}+\left\vert f^{\prime }\left( b\right) \right\vert ^{q}}{2}%
\right) ^{1/q}
\end{equation*}%
where $\alpha >0$ and $1/p+1/q=1.$
\end{theorem}

\begin{proof}
Using Lemma \ref{2.3}, H\"{o}lder's inequality, (\ref{2-4c}) and the
convexity of $\left\vert f^{\prime }\right\vert ^{q}$, it follows that%
\begin{eqnarray}
&&\left\vert \left( \frac{f(a)+f(b)}{2}\right) \left[ J_{a+}^{\alpha
}g(b)+J_{b-}^{\alpha }g(a)\right] -\left[ J_{a+}^{\alpha }\left( fg\right)
(b)+J_{b-}^{\alpha }\left( fg\right) (a)\right] \right\vert  \notag \\
&\leq &\frac{1}{\Gamma (\alpha )}\left( \int_{a}^{b}\left\vert
\int_{t}^{a+b-t}\left( b-s\right) ^{\alpha -1}g(s)ds\right\vert dt\right)
^{1-1/q}\left( \int_{a}^{b}\left\vert \int_{t}^{a+b-t}\left( b-s\right)
^{\alpha -1}g(s)ds\right\vert \left\vert f^{\prime }\left( t\right)
\right\vert ^{q}dt\right) ^{1/q}  \notag \\
&\leq &\frac{1}{\Gamma (\alpha )}\left[ \int_{a}^{\frac{a+b}{2}}\left(
\int_{t}^{a+b-t}\left\vert \left( b-s\right) ^{\alpha -1}g(s)\right\vert
ds\right) dt+\int_{\frac{a+b}{2}}^{b}\left( \int_{a+b-t}^{t}\left\vert
\left( b-s\right) ^{\alpha -1}g(s)\right\vert ds\right) dt\right] ^{1-1/q} 
\notag \\
&&\times \left[ \int_{a}^{\frac{a+b}{2}}\left( \int_{t}^{a+b-t}\left\vert
\left( b-s\right) ^{\alpha -1}g(s)\right\vert ds\right) \left\vert f^{\prime
}\left( t\right) \right\vert ^{q}dt+\int_{\frac{a+b}{2}}^{b}\left(
\int_{a+b-t}^{t}\left\vert \left( b-s\right) ^{\alpha -1}g(s)\right\vert
ds\right) \left\vert f^{\prime }\left( t\right) \right\vert ^{q}dt\right]
^{1/q}  \notag
\end{eqnarray}%
\begin{eqnarray}
&\leq &\frac{2^{1-1/q}\left\Vert g\right\Vert _{\infty }}{\left( b-a\right)
^{1/q}\Gamma (\alpha +1)}\left( \frac{\left( b-a\right) ^{\alpha +1}}{\alpha
+1}\left[ 1-\frac{1}{2^{\alpha }}\right] \right) ^{1-1/q}  \notag \\
&&\times \left\{ \int_{a}^{\frac{a+b}{2}}\left[ \left( b-t\right) ^{\alpha
}-\left( t-a\right) ^{\alpha }\right] \left( \left( b-t\right) \left\vert
f^{\prime }\left( a\right) \right\vert ^{q}+\left( t-a\right) \left\vert
f^{\prime }\left( b\right) \right\vert ^{q}\right) dt\right.  \notag \\
&&\left. +\int_{\frac{a+b}{2}}^{b}\left[ \left( t-a\right) ^{\alpha }-\left(
b-t\right) ^{\alpha }\right] \left( \left( b-t\right) \left\vert f^{\prime
}\left( a\right) \right\vert ^{q}+\left( t-a\right) \left\vert f^{\prime
}\left( b\right) \right\vert ^{q}\right) dt\right\} ^{1/q}  \label{2-5a}
\end{eqnarray}%
where it is easily seen that%
\begin{eqnarray*}
&&\int_{a}^{\frac{a+b}{2}}\left( \int_{t}^{a+b-t}\left( b-s\right) ^{\alpha
-1}ds\right) dt+\int_{\frac{a+b}{2}}^{b}\left( \int_{a+b-t}^{t}\left(
b-s\right) ^{\alpha -1}ds\right) dt \\
&=&\frac{2\left( b-a\right) ^{\alpha +1}}{\alpha \left( \alpha +1\right) }%
\left[ 1-\frac{1}{2^{\alpha }}\right] .
\end{eqnarray*}

Hence, if we use (\ref{2-4e}) and (\ref{2-4f}) in (\ref{2-5a}), we obtain
the desired result. This completes the proof.
\end{proof}

We can state another inequality for $q>1$ as follows:

\begin{theorem}
\label{2.6}Let $f:I\subseteq \mathbb{R\rightarrow R}$ be a differentiable
mapping on $I^{\circ }$ and $f^{\prime }\in L\left[ a,b\right] $ with $a<b$.
If $\left\vert f^{\prime }\right\vert ^{q},q>1,$ is convex on $\left[ a,b%
\right] $ and $g:\left[ a,b\right] \mathbb{\rightarrow R}$ is continuous and
symmetric to $(a+b)/2$, then the following inequalities for fractional
integrals hold:

(i)%
\begin{eqnarray}
&&\left\vert \left( \frac{f(a)+f(b)}{2}\right) \left[ J_{a+}^{\alpha
}g(b)+J_{b-}^{\alpha }g(a)\right] -\left[ J_{a+}^{\alpha }\left( fg\right)
(b)+J_{b-}^{\alpha }\left( fg\right) (a)\right] \right\vert  \notag \\
&\leq &\frac{2^{1/p}\left\Vert g\right\Vert _{\infty }\left( b-a\right)
^{\alpha +1}}{(\alpha p+1)^{1/p}\Gamma (\alpha +1)}\left( 1-\frac{1}{%
2^{\alpha p}}\right) ^{1/p}\left( \frac{\left\vert f^{\prime }\left(
a\right) \right\vert ^{q}+\left\vert f^{\prime }\left( b\right) \right\vert
^{q}}{2}\right) ^{1/q}  \label{2-6}
\end{eqnarray}%
with $\alpha >0.$

(ii)%
\begin{eqnarray}
&&\left\vert \left( \frac{f(a)+f(b)}{2}\right) \left[ J_{a+}^{\alpha
}g(b)+J_{b-}^{\alpha }g(a)\right] -\left[ J_{a+}^{\alpha }\left( fg\right)
(b)+J_{b-}^{\alpha }\left( fg\right) (a)\right] \right\vert  \notag \\
&\leq &\frac{\left\Vert g\right\Vert _{\infty }\left( b-a\right) ^{\alpha +1}%
}{(\alpha p+1)^{1/p}\Gamma (\alpha +1)}\left( \frac{\left\vert f^{\prime
}\left( a\right) \right\vert ^{q}+\left\vert f^{\prime }\left( b\right)
\right\vert ^{q}}{2}\right) ^{1/q}  \label{2-7}
\end{eqnarray}%
for $0<\alpha \leq 1.$ Where $1/p+1/q=1.$
\end{theorem}

\begin{proof}
(i) Using Lemma \ref{2.3}, H\"{o}lder's inequality, (\ref{2-4c}) and the
convexity of $\left\vert f^{\prime }\right\vert ^{q}$, it follows that%
\begin{eqnarray*}
&&\left\vert \left( \frac{f(a)+f(b)}{2}\right) \left[ J_{a+}^{\alpha
}g(b)+J_{b-}^{\alpha }g(a)\right] -\left[ J_{a+}^{\alpha }\left( fg\right)
(b)+J_{b-}^{\alpha }\left( fg\right) (a)\right] \right\vert \\
&\leq &\frac{1}{\Gamma (\alpha )}\left( \int_{a}^{b}\left\vert
\int_{t}^{a+b-t}\left( b-s\right) ^{\alpha -1}g(s)ds\right\vert
^{p}dt\right) ^{1/p}\left( \int_{a}^{b}\left\vert f^{\prime }(t)\right\vert
^{q}dt\right) ^{1/q}. \\
&\leq &\frac{\left\Vert g\right\Vert _{\infty }}{\Gamma (\alpha +1)}\left(
\int_{a}^{\frac{a+b}{2}}\left[ \left( b-t\right) ^{\alpha }-\left(
t-a\right) ^{\alpha }\right] ^{p}dt+\int_{\frac{a+b}{2}}^{b}\left[ \left(
t-a\right) ^{\alpha }-\left( b-t\right) ^{\alpha }\right] ^{p}dt\right)
^{1/p} \\
&&\times \left( \int_{a}^{b}\left( \frac{b-t}{b-a}\left\vert f^{\prime
}\left( a\right) \right\vert ^{q}+\frac{t-a}{b-a}\left\vert f^{\prime
}\left( b\right) \right\vert ^{q}\right) dt\right) ^{1/q}
\end{eqnarray*}%
\begin{eqnarray}
&=&\frac{\left\Vert g\right\Vert _{\infty }\left( b-a\right) ^{\alpha +1}}{%
\Gamma (\alpha +1)}\left( \int_{0}^{\frac{1}{2}}\left[ \left( 1-t\right)
^{\alpha }-t^{\alpha }\right] ^{p}dt+\int_{\frac{1}{2}}^{1}\left[ t^{\alpha
}-\left( 1-t\right) ^{\alpha }\right] ^{p}dt\right) ^{1/p}  \notag \\
&&\times \left( \frac{\left\vert f^{\prime }\left( a\right) \right\vert
^{q}+\left\vert f^{\prime }\left( b\right) \right\vert ^{q}}{2}\right) ^{1/q}
\label{2-6a}
\end{eqnarray}%
\begin{eqnarray*}
&\leq &\frac{\left\Vert g\right\Vert _{\infty }\left( b-a\right) ^{\alpha +1}%
}{\Gamma (\alpha +1)}\left( \int_{0}^{\frac{1}{2}}\left[ \left( 1-t\right)
^{\alpha p}-t^{\alpha p}\right] dt+\int_{\frac{1}{2}}^{1}\left[ t^{\alpha
p}-\left( 1-t\right) ^{\alpha p}\right] dt\right) ^{1/p} \\
&&\times \left( \frac{\left\vert f^{\prime }\left( a\right) \right\vert
^{q}+\left\vert f^{\prime }\left( b\right) \right\vert ^{q}}{2}\right) ^{1/q}
\end{eqnarray*}%
\begin{equation*}
\leq \frac{\left\Vert g\right\Vert _{\infty }\left( b-a\right) ^{\alpha +1}}{%
\Gamma (\alpha +1)}\left( \frac{2}{\alpha p+1}\left[ 1-\frac{1}{2^{\alpha p}}%
\right] \right) ^{1/p}\left( \frac{\left\vert f^{\prime }\left( a\right)
\right\vert ^{q}+\left\vert f^{\prime }\left( b\right) \right\vert ^{q}}{2}%
\right) ^{1/q}.
\end{equation*}%
Here we use%
\begin{equation*}
\left[ \left( 1-t\right) ^{\alpha }-t^{\alpha }\right] ^{p}\leq \left(
1-t\right) ^{\alpha p}-t^{\alpha p}
\end{equation*}%
for $t\in \left[ 0,1/2\right] $ and 
\begin{equation*}
\left[ t^{\alpha }-\left( 1-t\right) ^{\alpha }\right] ^{p}\leq t^{\alpha
p}-\left( 1-t\right) ^{\alpha p}
\end{equation*}%
for $t\in \left[ 1/2,1\right] $, which follows from 
\begin{equation*}
\left( A-B\right) ^{q}\leq A^{q}-B^{q},
\end{equation*}%
for any $A\geq B\geq 0$ and $q\geq 1$. Hence the inequality (\ref{2-6}) is
proved.

(ii) The inequality (\ref{2-7}) is easily proved using (\ref{2-6a}) and
Lemma \ref{1.6}.
\end{proof}

\begin{remark}
In Theorem \ref{2.6}, if we take $\alpha =1$, then equality (\ref{2-7})
becomes equality in \cite[Corollary 13]{WZZ13}.
\end{remark}

\end{document}